\documentclass[12pt]{amsart}
\usepackage{latexsym,rawfonts}
\usepackage{amsfonts,amssymb}
\usepackage{amsmath,amsthm}

\usepackage[plainpages=false]{hyperref}
\usepackage{graphicx}
\usepackage{color}

\newtheorem{theorem}{Theorem}
\newtheorem{corollary}{Corollary}

\newtheorem{lemma}{Lemma}

\theoremstyle{definition}

\newcommand{\R}{\mathbb{R}}

\newcommand{\dd}{\mathop{}\!\mathrm{d}}

\newcommand{\set}[1]{\left\{#1\right\}}

\newcommand{\pd}{\partial}

\newcommand{\uS}{\mathbb{S}^{n-1}}

\newcommand{\OM}{Orlicz-Minkowski }

\newcommand{\beq}{\begin{equation}}
\newcommand{\eeq}{\end{equation}}
\newcommand{\beqs}{\begin{eqnarray*}}
\newcommand{\eeqs}{\end{eqnarray*}}
\newcommand{\beqn}{\begin{eqnarray}}
\newcommand{\eeqn}{\end{eqnarray}}

\begin{document}

\title{Deforming a convex hypersurface by anisotropic curvature flows}

\author{HongJie Ju}
\address{HongJie Ju: School of Science,Beijing University of Posts and Telecommunications, Beijing 100876, P.R. China}
\email{hjju@bupt.edu.cn}

\author{BoYa Li}

\address{BoYa Li: School of Mathematics and Statistics, Beijing Technology and Business University, Beijing 100048, P.R. China}

\author{ YanNan Liu}

\address{YanNan Liu: School of Mathematics and Statistics, Beijing Technology and Business University, Beijing 100048, P.R. China}
\email{liuyn@th.btbu.edu.cn}

\thanks{This work was supported by Natural Science Foundation of China (11871432, 11871102) and Beijing Natural Science Foundation (1172005).}

\date{}

\begin{abstract}
In this paper,  we consider a fully nonlinear curvature flow of a convex hypersurface in the
Euclidean $n$-space. This flow  involves $k$-th  elementary symmetric function for principal curvature radii and a function of
support function.
Under some appropriate assumptions, we prove the long-time existence and
convergence of this flow.
As an application, we give the existence  of smooth solutions to  the Orlicz-Christoffel-Minkowski problem.
\end{abstract}




\maketitle
\vskip4ex

\section{Introduction}
Let $M_{0}$ be a smooth, closed, strictly convex hypersurface in the Euclidean
space $\R^{n}$, which encloses the origin and is given by a smooth embedding
$X_{0}: \mathbb{S}^{n-1} \rightarrow \R^{n}$.
Consider a family of closed hypersurfaces $\set{M_t}$ with $M_t=X(\uS,t)$, where $X:
\mathbb{S}^{n-1}\times[0,T) \rightarrow \R^{n}$ is a smooth map satisfying the
following initial value problem:
\begin{equation}\label{feq}
  \begin{split}
    \frac{\pd X}{\pd t} (x,t)
    &= \frac{1}{f(\nu)}\sigma_{k}(x,t) \varphi(\langle X, \nu \rangle)\langle X, \nu \rangle \eta(t) \nu - X,\\
    X(x,0) &= X_{0}(x).
  \end{split}
\end{equation}
Here $f$ is a given positive and smooth function on the unit sphere $\uS$, $\nu$ is the unit outer
normal vector of  $M_t$ at the point $X(x,t)$.  $\langle \cdot,\cdot \rangle$ is the
standard inner product in $\R^n$, $\varphi$ is a positive smooth function
defined in $(0,+\infty)$, $\eta$ is a scalar function to be specified later, and $T$ is the maximal time
for which the solution exists. We use $\{e_{ij}\}, 1 \leq i, j \leq n-1$ and $\nabla$ for the standard metric and Levi-Civita connection of  $\uS$ respectively. Principal radii
of curvature are the eigenvalues of the matrix
\beqs b_{ij} := \nabla_{i}\nabla_{j}h  + e_{ij}h \eeqs
with respect to $\{e_{ij}\}$.
$\sigma_{k}(x,t)$ is the
$k$-th  elementary symmetric function for principal curvature radii of $M_{t}$ at $X(x,t)$
and $k$ is an integer with $1\leq k \leq n-1$. In this paper,
$\sigma_{k}$ is normalized so that $\sigma_{k} (1, \ldots, 1) = 1$.

Geometric flows with speed of symmetric polynomial of the principal curvature radii of the hypersurface have been  extensively studied, see e.g. \cite{Urbas1991,ChowTsai1997,Gerhardt2014,Xia2016}.

On the other hand,  anisotropic curvature flows provide
alternative methods to prove the existences of elliptic PDEs arising from  convex geometry,  see e.g. \cite{BIS2019,Chen&Huang&Zhao2018,ChenLi,ChouWang2000,Ivaki2019,LiShengWang2020,LiuLuTrans,LiuLu2020,ShengYi2020}.

A positive homothetic self-similar solution of \eqref{feq}, if exists, is a solution to
 the following fully nonlinear equation
\begin{equation} \label{eq}
  c\, \varphi(h) \sigma_{k}(x) =f(x) \text{ on } \uS
\end{equation}
for some positive constant $c$.
Here $h$ is the support function defined on $\uS$. We are concerned the existence of smooth solutions for equation \eqref{eq}.

When $k = n-1$, equation \eqref{eq} is just the smooth case of \OM problem. The \OM problem is a basic problem in the Orlicz-Brunn-Minkowski theory in
convex geometry. It is  a generalization of the classical Minkowski problem which
asks what are the necessary and sufficient conditions for
a Borel measure on the unit sphere $\uS$ to be a multiple of the Orlicz surface area
measure of a convex body in $\R^n$. In \cite{HLYZ2010},  Haberl,
Lutwak, Yang \& Zhang studied the even case of the \OM problem. After that,
the \OM problem  attracted great attention from many scholars,
see for example \cite{GHWXY2019,GHWXY2020,HuangActa2016,JianLu2019,SunLong2015,XiJinLeng2014,ZouXiong2014}.

When $\varphi(s)=s^{1-p}, k=n-1$, Eq. \eqref{eq} reduces to the $L_p$-Minkowski
problem, which has been extensively studied, see e.g. \cite{BLYZ2013, ChouWang2006,YeLiWang2016,HuangLiuXu2015,HLYZ2005,JianLuWang2015,JianLuZhu2016,Lu2018,LuWang2013,Lutwak1993,Zhu2014}.

When $ 1\leq k< n-1$, Eq. \eqref{eq} is  so-called the Orlicz-Christoffel-Minkowski problem.
For $\varphi(s) = s^{1-p}, 1\leq k< n-1$, Eq. \eqref{eq}
is known as the $L_{p}$-Christoffel-Minkowski problem and is the  classical Christoffel-Minkowski problem for $p=1$.   Under a sufficient condition on the prescribed function, existence of solution for the classical Christoffel-Minkowski problem was given  in \cite{GuanMa2003}.

The $L_{p}$-Christoffel-Minkowski problem is related to the problem of prescribing $k$-th $p$-area measures. Hu, Ma \& Shen in  \cite{HMS2004} proved the existence of convex solutions to the $L_{p}$-Christoffel-Minkowski problem  for $ p \geq k+1$ under appropriate conditions. Using the methods of geometric flows, Ivaki in \cite{Ivaki2019} and then  Sheng \& Yi in \cite{ShengYi2020} also gave  the existence of smooth convex solutions to the $L_{p}$-Christoffel-Minkowski problem  for $ p \geq k+1$.
In case $1 <  p < k+1$, Guan \& Xia in  \cite{GuanXia2018} established the existence of convex body with prescribed $k$-th even $p$-area  measures.

In this paper, we study the long-time existence and convergence of flow \eqref{feq} for strictly convex hypersurfaces and   the
existence of smooth solutions to   the Orlicz-Christoffel-Minkowski problem \eqref{eq}.

The scalar function $\eta(t)$ in \eqref{feq} is usually used to  keep $M_t$ normalized in a certain sense, see for examples \cite{Chen&Huang&Zhao2018,Ivaki2019,ShengYi2020}.
In this paper, $\eta$ is  given by
\begin{equation}\label{eta}
  \eta(t) = \frac{\int_{\uS}hf(x)/\varphi(h)\dd x}{\int_{\uS}h\sigma_{k}\dd x},
\end{equation}
where $h(\cdot,t)$ is the support
function of the convex hypersurface $M_t$. It will be proved in section 2 that  $\int_{\uS}h\sigma_{k}\dd x$ is
non-decreasing along the flow under this choice of $\eta$.

To obtain the long-time existence of flow \eqref{feq}, we need some constraints on $\varphi$.

${\rm \bf (A)}$: $\varphi(s)$ is a  positive and continuous function defined in $(0,+\infty)$  such that  $\varphi > \alpha s^{-k-\varepsilon}$ for some positive constants $\varepsilon$ and $\alpha$ for  $s$ near $0$ and $\phi(s) = \int^{s}_{0} \frac{1}{\varphi(\tau)}d\tau$ is unbounded as $s \rightarrow +\infty$. Here $k$  is the order of $\sigma_{k}$.

The main results of this paper are stated as  follows.

\begin{theorem} \label{thm1}
  Assume $M_{0}$ is a smooth, closed and strictly convex hypersurface in $\R^{n}$. Suppose $k$ is an integer with $1 \leq k < n-1$ and   $\varphi\in
  C^\infty(0,+\infty)$ satisfying {\bf (A)}. Moreover, for any $s > 0$,
  \beqs
 \frac{\pd}{\pd s}\left(s\frac{\pd}{\pd s}\left(\log\varphi(s)\right) \right) \geq 0 \quad\text{and } \quad -a \leq s\frac{\pd}{\pd s}\left(\log\varphi(s)\right)  \leq -1, \eeqs
where  $a$ is a positive constant.
Suppose $f$ is a smooth function on $\uS$ such that $$(k  + 1)f^{-\frac{1}{k + a}}e_{ij} + (k + a)\nabla_{i}\nabla_{j}(f^{-\frac{1}{k + a}})$$ is positive definite.
 Then flow \eqref{feq}
  has a unique smooth solution $M_{t}$ for all time $t > 0$.
  Moreover, when $t\to\infty$, a subsequence of $M_{t}$ converges in $C^{\infty}$ to a smooth,
  closed,  strictly convex hypersurface, whose support
  function is a smooth  solution to equation \eqref{eq} for some positive
  constant $c$.
\end{theorem}

When $f \equiv 1$, we have the following result.
\begin{theorem} \label{thm2}
  Assume $M_{0}$ is a smooth, closed and strictly convex hypersurface in $\R^{n}$.
  If $f \equiv 1$,  $\varphi\in
  C^\infty(0,+\infty)$ satisfying {\bf (A)}, and $k$ is an integer with $1 \leq k < n-1$. Moreover, for any $s > 0$,
   \beqs
 \frac{\pd}{\pd s}\left(s\frac{\pd}{\pd s}(\log\varphi(s)) \right) \geq 0 \quad\text{and } \quad s\frac{\pd}{\pd s}(\log\varphi(s))  \leq -1. \eeqs
 Then flow \eqref{feq}
  has a unique smooth solution $M_{t}$ for all time $t > 0$.
  Moreover, when $t\to\infty$, a subsequence of $M_{t}$ converges in $C^{\infty}$ to a smooth,
  closed,  strictly convex hypersurface, whose support
  function is a smooth solution to equation \eqref{eq} for some positive
  constant $c$.
\end{theorem}

  As an application, we have
\begin{corollary} \label{coreq}
  Under the assumptions of Theorem \ref{thm1} or Theorem \ref{thm2},
  there exists a smooth  solution to equation \eqref{eq} for some
  positive constant $c$.
\end{corollary}

From the proof of Lemma \ref{C2-estimate-2} in section 3, we will see if $\dfrac{\varphi'(s)s}{\varphi(s)} = a_{0}$ for some negative constant $a_{0}$, then convexity condition on $f$ reduces to  $f^{-\frac{1}{k - a_{0}}}e_{ij} + \nabla_{j}\nabla_{j}(f^{-\frac{1}{k - a_{0}}})$ being positive definite.
Hence when $\varphi(s) = s^{1-p}$ for $p \geq k + 1 $ with above condition on $f$, our conclusion recovers the existence results to the $L_{p}$-Christoffel-Minkowski problem which have been obtained in  \cite{HMS2004},\cite{Ivaki2019} and \cite{ShengYi2020}.

This paper is organized as follows.
In section 2, we give some basic knowledge about the flow \eqref{feq} and evolution equations of some geometric quantities.
In section 3, the long-time existence of flow \eqref{feq} will be obtained.
First, under assumption {\bf (A)},  uniform positive
upper and lower bounds for support functions of $\set{M_t}$ are derived.
Based on the bounds of support functions, we obtain the uniform bounds of  principal curvatures
by constructing proper auxiliary functions.
The long-time existence of flow \eqref{feq} then follows by standard arguments.
In section 4, by considering a related geometric functional, we prove that
a subsequence of $\set{M_t}$ converges to a smooth solution to equation
\eqref{eq}, completing the proofs of Theorem \ref{thm1} and Theorem \ref{thm2}.

\section{Preliminaries}
Let $\R^n$ be the $n$-dimensional Euclidean space, and $\uS$ be the unit sphere
in $\R^n$.
Assume $M$ is a smooth closed strictly convex hypersurface in $\R^{n}$.
Without loss of generality, we may assume that $M$ encloses the origin.
The support function $h$ of $M$ is defined as
\begin{equation*}
h(x) := \max_{y\in M} \langle y,x \rangle, \quad \forall x\in\uS,
\end{equation*}
where $\langle \cdot,\cdot \rangle$ is the standard inner product in $\R^n$.

Denote the Gauss map of $M$  by $\nu_M$.
Then $M$ can be parametrized by the inverse Gauss map $X :
\mathbb{S}^{n-1}\rightarrow M$ with $X(x) =\nu_M^{-1}(x)$.
The support function $h$ of $M$ can be computed by
\begin{equation} \label{h}
  h(x) = \langle x, X(x)\rangle, \indent x \in \mathbb{S}^{n-1}.
\end{equation}
Note that $x$ is just the unit outer normal of $M$ at $X(x)$.
Differentiating \eqref{h}, we have
\begin{equation*}
  \nabla_{i} h = \langle \nabla_{i}x, X(x)\rangle + \langle x, \nabla_{i}X(x)\rangle.
\end{equation*}
  Since $\nabla_{i}X(x)$ is tangent to $M$ at $X(x)$, we have
\begin{equation*}
  \nabla_{i} h = \langle \nabla_{i}x, X(x)\rangle.
\end{equation*}
It follows that
\begin{equation}\label{Xh}
  X(x) = \nabla h + hx.
\end{equation}

By differentiating \eqref{h} twice, the  second fundamental form $A_{ij}$   of $M$
can be computed in terms of the support function, see for example \cite{Urbas1991},
\begin{equation}
\label{A} A_{ij} =  \nabla_{ij}h + he_{ij},
\end{equation}
where $\nabla_{ij} = \nabla_{i}\nabla_{j}$ denotes the second order covariant derivative with respect to $e_{ij}$.
The  induced metric matix $g_{ij}$ of $M$ can be derived by Weingarten's formula,
\begin{equation}
  \label{g}
  e_{ij} = \langle \nabla_{i}x, \nabla_{j}x\rangle  = A_{im}A_{lj}g^{ml}.
\end{equation}
The principal radii of curvature are the eigenvalues of the matrix $b_{ij} =
A^{ik}g_{jk}$.
When considering a smooth local orthonormal frame on $\uS$, by virtue of
\eqref{A} and \eqref{g}, we have
\begin{equation}
  \label{radii}
  b_{ij} = A_{ij} = \nabla_{ij}h + h\delta_{ij}.
\end{equation}
We will use
$b^{ij}$ to denote the inverse matrix of $b_{ij}$.

From the evolution equation of $X(x,t)$ in flow \eqref{feq}, we derive the
evolution equation of the corresponding support function $h(x,t)$:
\begin{equation}\label{seq}
\frac{\pd h(x,t)}{\pd t} = \frac{1}{f(x)}\sigma_{k}(x,t)\varphi(h)h(x,t)\eta(t) - h(x,t).
\end{equation}

The radial function $\rho$ of $M$ is given by
\begin{equation*}
\rho(u) :=\max\set{\lambda>0 : \lambda u\in M}, \quad\forall\ u\in\uS.
\end{equation*}
Note that $\rho(u)u\in M$.

From $\eqref{Xh}$, $u$ and $x$ are related by
\begin{equation}
  \label{rs}
  \rho(u)u = \nabla h(x) + h(x)x
\end{equation}
and
\beqs \rho^{2} = |\nabla h|^{2} + h^{2}.\eeqs

 In the rest of the paper, we take a local orthonormal frame $\{e_{1}, \cdots, e_{n-1}\}$ on $\mathbb{S}^{n-1} $ such that the standard metric on $\mathbb{S}^{n-1} $ is $\{\delta_{ij}\}$.  Double indices
 always mean to sum from $1$ to $n-1$. We  denote partial derivatives $ \dfrac{\partial \sigma_{k}}{\partial b_{ij}}$ and $ \dfrac{\partial^{2} \sigma_{k}}{\partial b_{pq}\partial b_{mn}}$ by $\sigma_{k}^{ij}$ and $\sigma_{k}^{pq,mn}$  respectively.  For convenience, we also write
\begin{equation*}
\begin{split}
N &= \frac{1}{f(x)}\varphi(h)h, \\
F & = N\sigma_{k}\eta(t).
\end{split}
\end{equation*}

Now, we can prove that the mixed volume  $\int_{\uS}h(x,t)\sigma_{k}(x,t) \dd x$ is non-decreasing along the flow \eqref{feq}.
 \begin{lemma}\label{lem1}
 $\int_{\uS}h(x,t)\sigma_{k}(x,t) \dd x $ is non-decreasing along the flow \eqref{feq}.
 \end{lemma}
 \begin{proof}   According to the evolution equation of $h$ in \eqref{seq}, we get
\begin{equation*}
\begin{split}
\partial_{t}\sigma_{k} & =\sigma_{k}^{ij}\partial_{t}(\nabla_{ij}h + \delta_{ij}h) \\
&=  \sigma_{k}^{ij}\nabla_{ij}(\partial_{t}h) +  \sigma_{k}^{ij}\delta_{ij}\partial_{t}h\\
&= \sigma_{k}^{ij}\nabla_{ij}F - \sigma_{k}^{ij}\nabla_{ij}h + \sigma_{k}^{ij}\delta_{ij}F - \sigma_{k}^{ij}\delta_{ij}h\\
&= \sigma_{k}^{ij}\nabla_{ij}F + \sigma_{k}^{ij}\delta_{ij}F - k\sigma_{k},
\end{split}
\end{equation*}
the last equality holds because $\sigma_{k}$ is homogeneous of degree $k$ and $\sigma_{k}^{ij}b_{ij} = k\sigma_{k}.$
Hence,
\begin{equation*}
\begin{split}
& \partial_{t}\int_{\uS}h\sigma_{k} \dd x\\
& =   \int_{\uS}(\partial_{t}\sigma_{k})h \dd x + \int_{\uS}\sigma_{k}\partial_{t}h\dd x\\
& = \int_{\uS}\left( h\sigma_{k}^{ij}\nabla_{ij}F + h\sigma_{k}^{ij}\delta_{ij}F - kh\sigma_{k}\right)\dd x +  \int_{\uS}F\sigma_{k}dx - \int_{\uS}h\sigma_{k}dx \\
& = (k+1)\int_{\uS}F\sigma_{k}\dd x  - (k+1)\int_{\uS}h\sigma_{k}\dd x + \int_{\uS}\left(h\sigma_{k}^{ij}\nabla_{ij}F - F\sigma_{k}^{ij}\nabla_{ij}h\right)\dd x\\
& = (k+1)\int_{\uS}F\sigma_{k}\dd x  - (k+1)\int_{\uS}h\sigma_{k}\dd x,
\end{split}
\end{equation*}
where in the last equality  we use the fact $\sum_{i}\nabla_{i}(\sigma_{k}^{ij}) = 0$.

By H\"older's Inequality, we have
\begin{equation*}
\begin{split}
 &\frac{1}{k+1}\partial_{t}\int_{\uS}h\sigma_{k} \dd x\\
  &= \int_{\uS}\frac{1}{f(x)}\sigma_{k}^{2}\varphi(h)h\eta\dd x - \int_{\uS}h\sigma_{k} \dd x\\
 & = \frac{1}{\int_{\uS}h\sigma_{k}\dd x}\left[\int_{\uS}\frac{1}{f(x)}\sigma_{k}^{2}\varphi(h)h\dd x\int_{\uS}\frac{h}{\varphi(h)}f(x)\dd x - \left(\int_{\uS}h\sigma_{k}\dd x\right)^{2}\right]\\
 & \geq 0,
\end{split}
\end{equation*}
and the equality holds if and only if
$$ c\varphi (h)\sigma_{k}(x) = f(x)$$
for some positive constant $c$.
\end{proof}

By the flow equation \eqref{feq}, we can derive evolution equations of some geometric quantities.

 \begin{lemma}\label{evolutions}
The following evolution equations hold along the flow \eqref{feq}.
\begin{equation*}
\begin{split}
& \partial_{t}b_{ij} -  N\eta(t)\sigma_{k}^{pq}\nabla_{pq}b_{ij} \\
 & = (k+1) N\eta(t)\sigma_{k}\delta_{ij} -  N\eta(t)\sigma_{k}^{pq}\delta_{pq}b_{ij} + N\eta(t)(\sigma_{k}^{ip}b_{jp} - \sigma_{k}^{jp}b_{ip})\\
& +  N\eta(t)\sigma_{k}^{pq,mn}\nabla_{j}b_{pq}\nabla_{i}b_{mn} + \eta(t)\left(\sigma_{k}\nabla_{ij}N   + \nabla_{j}\sigma_{k}\nabla_{i}N
+ \nabla_{i}\sigma_{k}\nabla_{j}N\right) - b_{ij}\\
& \partial_{t}b^{ij} -  N\eta(t)\sigma_{k}^{pq}\nabla_{pq}b^{ij} \\
 & =- (k+1) N\eta(t)\sigma_{k}b^{ip}b^{jp} +  N\eta(t)\sigma_{k}^{pq}\delta_{pq}b^{ij} - N\eta(t)b^{ip}b^{jq}(\sigma_{k}^{rp}b_{rq} - \sigma_{k}^{rq}b_{rp})\\
& -  N\eta(t)b^{il}b^{js}(\sigma_{k}^{pq,mn} + 2\sigma_{k}^{pm}b^{nq})\nabla_{l}b_{pq}\nabla_{s}b_{mn}\\
& -\eta(t)b^{ip}b^{jq}( \sigma_{k}\nabla_{ij}N  + \nabla_{j}\sigma_{k}\nabla_{i}N
+ \nabla_{i}\sigma_{k}\nabla_{j}N) + b^{ij}\\
 & \partial_{t}\left(\frac{\rho^{2}}{2}\right) - N\eta(t)\sigma_{k}^{ij}\nabla_{ij}\left(\frac{\rho^{2}}{2}\right)\\
& = (k+1)hN\eta(t)\sigma_{k} - \rho^{2} + \eta(t)\sigma_{k}\nabla_{i}h\nabla_{i}N -  N\eta(t)\sigma_{k}^{ij}b_{mi}b_{mj}.
\end{split}
 \end{equation*}
 \end{lemma}
 \begin{proof}
 From \eqref{feq},
 \begin{equation*}
\partial_{t}\nabla_{ij}h = \nabla_{ij}(\partial_{t}h) = \eta(t)(\sigma_{k}\nabla_{ij}N   + \nabla_{j}\sigma_{k}\nabla_{i}N
+ \nabla_{i}\sigma_{k}\nabla_{j}N) + N\eta(t)\nabla_{ij}\sigma_{k} - h_{ij},
\end{equation*}
where
\begin{equation*}
\nabla_{ij}\sigma_{k} = \sigma_{k}^{pq,mn}\nabla_{j}b_{pq}\nabla_{i}b_{mn} + \sigma_{k}^{pq}\nabla_{ij}b_{pq}.
\end{equation*}
By Gauss equation,
\begin{equation*}
\nabla_{ij}b_{pq} = \nabla_{pq}b_{ij} + \delta_{ij}\nabla_{pq}h - \delta_{pq}\nabla_{ij}h  + \delta_{iq}\nabla_{pj}h  - \delta_{pj}\nabla_{iq}h .
\end{equation*}
Hence
\begin{equation*}
\begin{split}
\partial_{t}h_{ij} & =  N\eta(t)\sigma_{k}^{pq}\nabla_{pq}b_{ij} + k N\eta(t)\sigma_{k}\delta_{ij} -  N\eta(t)\sigma_{k}^{pq}\delta_{pq}b_{ij} + N\eta(t)(\sigma_{k}^{ip}b_{jp} - \sigma_{k}^{jp}b_{ip})\\
& +  N\eta(t)\sigma_{k}^{pq,mn}\nabla_{j}b_{pq}\nabla_{i}b_{mn} + \eta(t)(\sigma_{k}\nabla_{ij}N  + \nabla_{j}\sigma_{k}\nabla_{i}N
+ \nabla_{i}\sigma_{k}\nabla_{j}N) - h_{ij}.
\end{split}
\end{equation*}
This together with $\eqref{radii}$ gives the evolution equation of $b_{ij}$.
The evolution equation of $b^{ij} $ then follows from
\begin{equation*}
\partial_{t}b^{ij} = -b^{im}b^{lj}\partial_{t}b_{ml}.
\end{equation*}
 For more details of computations about the evolution equations of $b_{ij}$ and $b^{ij}$, one can refer to \cite{ChowTsai1997,Urbas1991}.

Recalling that $ \rho^{2} = h^{2} + |\nabla h|^{2}$, we have
 \begin{equation*}
 \begin{split}
& \partial_{t}\left(\frac{\rho^{2}}{2}\right) - N\eta(t)\sigma_{k}^{ij}\nabla_{ij}\left(\frac{\rho^{2}}{2}\right)\\
& = h\partial_{t}h + \nabla_{i}h \nabla_{i}\partial_{t}h - N\eta(t)\sigma_{k}^{ij}\left(h\nabla_{ij}h + \nabla_{i}h\nabla_{j}h + \nabla_{m}h\nabla_{j}\nabla_{mi}h + \nabla_{mi}h\nabla_{mj}h \right)\\
& = h\partial_{t}h + \nabla_{i}h\nabla_{i}(N\eta(t)\sigma_{k} - h) \\
&- N\eta(t)\sigma_{k}^{ij}\left[\nabla_{i}h\nabla_{j}h + \nabla_{m}h\nabla_{j}(b_{mi} - h\delta_{mi}) \right] - N\eta(t)\sigma_{k}^{ij}h(b_{ij} - h\delta_{ij}) \\
& - N\eta(t)\sigma_{k}^{ij}(b_{mi} - h\delta_{mi})(b_{mj} - h\delta_{mj})\\
& = (k+1)hN\eta(t)\sigma_{k} - \rho^{2} + \eta(t)\sigma_{k}\nabla_{i}h\nabla_{i}N -  N\eta(t)\sigma_{k}^{ij}b_{mi}b_{mj}.
\end{split}
 \end{equation*}
  \end{proof}

\section{The long-time existence of  the flow }

In this section, we will give a priori estimates about support functions and curvatures to
obtain the long-time existence of flow \eqref{feq} under assumptions of Theorem \ref{thm1} and Theorem \ref{thm2}.

In the rest of this paper, we  assume that $M_{0}$ is a smooth,
closed,  strictly convex hypersurface in $\R^{n}$ and
$h: \uS\times[0,T)\to \R$ is a smooth solution to the evolution equation \eqref{seq}
with the initial $h(\cdot,0)$ the support function of $M_0$.
Here $T$ is the maximal time for which the solution exists.
Let $M_t$ be the convex hypersurface determined by $h(\cdot,t)$, and
$\rho(\cdot,t)$ be the corresponding radial function.

We first give the uniform positive upper and lower bounds of $h(\cdot,t)$ and $\rho(u,t)$ for $t\in[0,T)$.

\begin{lemma}\label{C0-estimate} Let $h$ be a smooth  solution of $\eqref{seq}$ on $\mathbb{S}^{n-1} \times [0, T)$, $f$ be a positive, smooth function on $\uS$
and $\varphi\in
  C^\infty(0,+\infty)$ be a decreasing function satisfying ${\rm \bf (A)}$. Then
\beqn  \label{h1}\frac{1}{C} \leq  h(x,t) \leq C, \\
\label{rho1}\frac{1}{C} \leq  \rho(u,t) \leq C ,\eeqn
where $C$ is a positive constant independent of $t$.
\end{lemma}
\begin{proof}
Let $J(t) = \int_{\uS} \phi(h(x,t))f(x)\dd x$. We claim that $J(t)$ is unchanged along the flow $\eqref{feq}$. It is because
\begin{equation*}
\begin{split}
 J'(t) &= \int_{\uS}\phi'(h)\partial_{t} hf(x) \dd x \\
&=   \int_{\uS}\frac{f(x)}{\varphi(h)}\partial_{t} h \dd x \\
&= \int_{\uS}\frac{f(x)}{\varphi(h)}\left(\frac{1}{f(x)} \sigma_{k}(x)\varphi(h)h\eta(t) - h\right)\dd x\\
&= \int_{\uS}\sigma_{k}(x)h\eta(t)\dd x  - \int_{\uS}\frac{h}{\varphi(h)}f(x) \dd x\\
& = 0.
\end{split}
\end{equation*}

For each $t \in [0,T)$, suppose that the maximum of radial function $\rho(\cdot,t)$ is attained at some $u_{t} \in \uS$.
 Let
$$R_{t} = \max_{u \in \uS}\rho (u,t) = \rho (u_{t},t)$$ for some $u_{t}\in \uS$.
By the definition of support function, we have
\begin{equation*}
h(x,t)\geq R_t\langle x,u_t \rangle, \quad \forall x\in\uS.
\end{equation*}
Denote the hemisphere containing $u_t$ by $S_{u_{t}}^{+} =\set{x\in\uS : \langle x,u_t \rangle > 0}$. Since  $\phi'(h) = \frac{1}{\varphi(h)} > 0 $ implies that $\phi(h) $ is strictly increasing about $h$, we have
\begin{equation*}
\begin{split}
J(0) & = J(t)\\
& \geq \int_{S_{u_{t}}^{+}} \phi(h(x,t))f(x)\dd x \\
 &  \geq \int_{S_{u_{t}}^{+}}\phi(R_{t} \langle x,u_t \rangle)f(x)\dd x\\
 & \geq  f_{\min}\int_{S_{u_{t}}^{+}}\phi(R_{t} \langle x,u_t \rangle)\dd x\\
   &= f_{\min} \int_{S^{+}} \phi(R_tx_1) \dd x,
 \end{split}
\end{equation*}
 where $S^{+}  =\set{x\in \uS : x_1\ >  0}$.

Denote $S_1 =\set{x\in \uS : x_1\geq 1/2}$, then
\begin{equation*}
\begin{split}
  J(0)
  &\geq f_{\min} \int_{S_1} \phi(R_t/2) \dd x \\
  &= f_{\min} \phi(R_t/2) |S_1|,
\end{split}
\end{equation*}
which implies that $\phi(R_t/2)$ is uniformly bounded from above.
By assumption {\bf (A)}, $\phi(s)$ is strictly increasing and tends to $+\infty$
as $s\to+\infty$.
Thus $R_t$ is uniformly bounded from above.

Now we can prove that $\eta(t)$ has positive lower bound. Since mixed volumes are monotonic increasing, see \cite[page 282]{Schneider}, we
have for each $t \in [0,T)$
 \begin{equation*}
 h_{\min}^{k+1}(t) \leq \dfrac{ \int_{\uS}h\sigma_{k} \dd x }{\omega_{n-1}}\leq   h_{\max}^{k+1}(t),
\end{equation*}
here $h_{\min}(t) = \min_{x \in \uS} h(x,t)$ and $h_{\max}(t) = \max_{x \in \uS} h(x,t)$.

 This  together with Lemma \ref{lem1} and  the upper bound of $h$ implies that there exist  positive constants $c_{1}$ and $c_{2}$ such that
  \begin{equation*}
 \int_{\uS}h\sigma_{k} \dd x \leq c_{1}
\end{equation*}
and
 \begin{equation*}
 h_{\max}(t) \geq  c_{2}.
 \end{equation*}
Recalling the definition of  $\eta(t)  $ and noticing that $\dfrac{1}{\varphi(s)}$ is an increasing function,
we have
 \begin{equation*}
 \begin{split}
\eta(t) & = \dfrac{\int_{\uS}hf(x)/\varphi(h)\dd x}{\int_{\uS}h\sigma_{k}\dd x}\\
&\geq \frac{1}{c_{1}}\int_{S_{u_{t}}^{+}} hf(x)/\varphi(h)\dd x\\
  & \geq \frac{1}{c_{1}}\int_{S_{u_{t}}^{+}}R_{t} \langle x,u_t \rangle f_{\min}\frac{1}{\varphi(R_{t}\langle x,u_t \rangle)}\dd x\\
  & = \frac{1}{c_{1}}f_{\min}\int_{S^{+}}R_{t}x_{1}\frac{1}{\varphi(R_{t} x_{1})}\dd x\\
  & \geq \frac{1}{c_{1}}f_{\min}|S_{1}| \frac{1}{2}R_{t}\frac{1}{\varphi( \frac{1}{2}R_{t})}\\
  & \geq c_{3},
\end{split}
 \end{equation*}
where $c_{3}$ is a positive constant independent of $t$.

Suppose the minimum of $h(x,t)$ is attained at a point $(x_{t}, t)$. At $(x_{t}, t)$, $\nabla_{ij}h$ is non-negative. It follows that
\begin{equation*}
\sigma_{k}(x_{t}, t) \geq h_{\min}^{k}(t).
 \end{equation*}
Then in the sense of the lim inf of  difference quotient, see \cite{Hamilton86}, we have
\begin{equation*}
 \begin{split}
 \frac{\partial h_{\min}(t)}{\partial t} & \geq \frac{1}{f_{\max}}h_{\min}(t)[\eta(t)h_{\min}^{k}(t)\varphi(h_{\min}(t)) - f(x)]\\
  & \geq \frac{1}{f_{\max}}h_{\min}(t)[c_{3}h_{\min}^{k}(t)\varphi(h_{\min}(t))- f_{\max}].
\end{split}
 \end{equation*}
 If $\varphi(s) > \alpha s^{-k-\varepsilon} $ for some positive constant $\varepsilon$ for $s $ near $0$, then
 \begin{equation*}
 \begin{split}
 \frac{\partial h_{\min}(t)}{\partial t}
  & \geq \frac{1}{f_{\max}}h_{\min}(t)(h_{\min}^{-\varepsilon}(t)\alpha c_{3}- f_{\max}).
\end{split}
 \end{equation*}
 The right hand of the above inequality is positive for $h_{\min}(t)$ small enough and the lower bound of $h_{\min}(t)$ follows from the maximum principle in \cite{Hamilton86}.

\end{proof}

By the equality $\rho^{2} = h^{2} + |\nabla h|^{2}$, we can obtain the gradient estimate of support function from Lemma \ref{C0-estimate}.

\begin{corollary}\label{cor3.2}
  Under the assumptions of Lemma \ref{C0-estimate}, we have
\begin{equation*}
  |\nabla h(x,t)| \leq C, \quad \forall (x,t) \in \mathbb{S}^{n-1} \times [0, T),
\end{equation*}
where $C$ is a positive constant depending only on constants in Lemma \ref{C0-estimate}.
\end{corollary}

 The uniform bounds of $\eta(t)$ can be derived from Lemmas \ref{lem1} and \ref{C0-estimate}.
\begin{lemma}\label{eta-estimate}   Under the assumptions of Lemma \ref{C0-estimate}, $\eta(t)$ is uniformly bounded above and below from zero.
\end{lemma}
\begin{proof}
In term of the proof of Lemma \ref{C0-estimate}, $\eta(t)$ has uniform positive lower bound. From Lemma \ref{lem1}, we know that
$\int_{\uS}h\sigma_{k} \dd x$ is monotonic decreasing about $t$, which give a positive lower bound on $\int_{\uS}h\sigma_{k}\dd x$. This together  with
the uniform bounds of $h(x,t)$ implies that $\eta(t)$ is bounded from above.

\end{proof}

To obtain the long-time existence of the flow $\eqref{feq}$, we need to establish the uniform bounds on principal curvatures.
By Lemma \ref{C0-estimate}, for any $t\in[0,T)$, $h(\cdot,t)$
 always ranges within a bounded interval $I'=[1/C,C]$, where $C$ is the
 constant in Lemma \ref{C0-estimate}.
 First, we  give the estimates of $\sigma_{k}$.

\begin{lemma}\label{sigma-lower-estimate}   Under the assumptions of Lemma \ref{C0-estimate},
$$\sigma_{k}(x,t) \geq C, \quad \forall (x,t) \in \mathbb{S}^{n-1} \times [0, T),$$
where $C$ is a positive constant independent of $t$.
\end{lemma}

\begin{proof}
Consider the auxiliary function  $ Q = \log M - A\frac{\rho^{2}}{2}, $ where
 $M = N\sigma_{k} = \frac{1}{f(x)}\varphi(h)h\sigma_{k}$ and $A$ is a positive constant to be determined later.
The evolution equation of $M$ is given by
\begin{equation*}
 \begin{split}
\partial_{t}M & =   N \partial_{t}\sigma_{k} + \sigma_{k} \partial_{t}N\\
& = N(\sigma_{k}^{ij}\nabla_{ij}F + \sigma_{k}^{ij}\delta_{ij}F - k\sigma_{k}) + \frac{M}{h}\left(1+ \frac{\varphi'h}{\varphi} \right)(F - h)\\
& = N\sigma_{k}^{ij}\nabla_{ij}F  + N\sigma_{k}^{ij}\delta_{ij}F - kM + \frac{M^{2}}{h}\eta(t)\left(1+ \frac{\varphi'h}{\varphi}  \right) - M\left(1+ \frac{\varphi'h}{\varphi} \right)\\
& = N\eta(t)\sigma_{k}^{ij}\nabla_{ij}M  + MN\eta(t)\sigma_{k}^{ij}\delta_{ij} - M\left(k + 1 + \frac{\varphi'h}{\varphi} \right) + \frac{M^{2}}{h}\eta(t)\left(1+ \frac{\varphi'h}{\varphi}  \right).
\end{split}
 \end{equation*}
It is easy to compute
\beqs \begin{split}
\nabla_{i} Q &= \frac{\nabla_{i}M }{M} - A\nabla_{i}\left(\frac{\rho^{2}}{2}\right) ,\\
\nabla_{ij}Q &= \frac{\nabla_{ij}M }{M} - \frac{1}{M^{2}}\nabla_{i}M \nabla_{j}M  - A\nabla_{ij} \left(\frac{\rho^{2}}{2}\right).\end{split}\eeqs
Due to the evolution equation of $\dfrac{\rho^{2}}{2}$ in Lemma \ref{evolutions}, the evolution equation of $Q$ is
  \begin{equation*}
 \begin{split}
& \partial_{t}Q - N\eta(t)\sigma_{k}^{ij}\nabla_{ij}Q \\
& =  \frac{1}{M^{2}}N\eta(t)\sigma_{k}^{ij}\nabla_{i}M \nabla_{j}M  + N\eta(t)\sigma_{k}^{ij}\delta_{ij} - \left(k+1 + \frac{\varphi'h}{\varphi} \right) + \frac{M}{h}\eta(t)\left(1 + \frac{\varphi'h}{\varphi} \right)\\
& - (k+1)AhN\eta(t)\sigma_{k} + A\rho^{2} - A\eta(t)\sigma_{k}\nabla_{i}h \nabla_{i}N  + AN\eta(t)\sigma_{k}^{ij}b_{mi}b_{mj}.
\end{split}
 \end{equation*}
 For fixed $t$, at a point where $Q$ attains its spatial minimum, we have
  \begin{equation*}
 \begin{split}
 \partial_{t}Q & \geq  A\rho^{2}- \left(k+1 + \frac{\varphi'h}{\varphi} \right) + \frac{M}{h}\eta(t)\left(1 + \frac{\varphi'h}{\varphi} \right) \\
 & - (k+1)AhN\eta(t)\sigma_{k}  - A\eta(t)\sigma_{k}\nabla_{i}h \nabla_{i}N\\
 & = \frac{1}{2}A\rho^{2}- \left(k+1 + \frac{\varphi'h}{\varphi} \right) + \frac{1}{h}e^{Q + A\frac{\rho^{2}}{2}}\eta(t)\left(1 + \frac{\varphi'h}{\varphi}  \right)\\
 &+ AN\eta(t)\sigma_{k}\left( \frac{\rho^{2}}{2e^{Q + A\frac{\rho^{2}}{2}} \eta(t)} - h(k+1) - \frac{1}{N}\nabla_{i}h \nabla_{i}N\right).
\end{split}
 \end{equation*}
 Now we choose $A > \dfrac{2}{\rho^{2}}(k+1)$. Notice that $\varphi$ is a monotonic decreasing, positive function and we have obtained  uniform bounds of $h, \rho, |\nabla h|$ and $\eta(t)$. If $Q$ is negatively large enough, the right-hand side is positive and the lower bound of $Q$ follows.
\end{proof}

\begin{lemma}\label{sigma-upper-estimate}   Under the assumptions of Lemma \ref{C0-estimate},
 $$\sigma_{k} \leq C, \quad \forall (x,t) \in \mathbb{S}^{n-1} \times [0, T),$$
where $C$ is a positive constant independent of $t$.
\end{lemma}
\begin{proof}
By Lemma \ref{C0-estimate}, there exists a positive constant $B$ such that
\beqs B < \rho^{2} < \frac{1}{B}\eeqs for all $t > 0.$
Define
\beqs P(x,t) = \dfrac{\varphi\sigma_{k}}{f(1 - \frac{ B\rho^{2} }{2})} = \frac{M}{h}\frac{1}{1 - \frac{B \rho^{2} }{2}}.\eeqs
By the evolution equation of $M$ in Lemma \ref{sigma-lower-estimate},
we have
\begin{equation*}
\partial_{t}\frac{M}{h}  -  N\eta(t)\sigma_{k}^{ij}\nabla_{ij}\frac{M}{h} =  - \frac{M}{h} \left(k  + \frac{\varphi'h}{\varphi} \right)
 +  \frac{M^{2}}{h^{2}}\eta(t)\left(k + \frac{\varphi'h}{\varphi} \right) + \frac{2N}{h}\eta(t)\sigma_{k}^{ij}\nabla_{i}h\nabla_{j}\frac{M}{h}.
 \end{equation*}
Hence
\begin{equation*}
\begin{split}
& \partial_{t}P -  N\eta(t)\sigma_{k}^{ij}\nabla_{ij}P\\
& = \frac{1}{1 - \frac{ B\rho^{2} }{2}}\left[- \frac{M}{h} \left(k  + \frac{\varphi'h}{\varphi} \right)
 +  \frac{M^{2}}{h^{2}}\eta(t)\left(k + \frac{\varphi'h}{\varphi}  \right) + \frac{2N}{h}\eta(t)\sigma_{k}^{ij}\nabla_{i}h\nabla_{j}\frac{M}{h}\right]\\
 &+ \frac{MB}{h(1 - \frac{ B\rho ^{2}}{2}) ^{2}} \left[ (k+1)Nh\eta(t)\sigma_{k}
 - \rho^{2} + \eta(t)\sigma_{k}\nabla_{i}h\nabla_{i}N - N\eta(t)\sigma_{k}^{ij}b_{mi}b_{mj}\right]\\
 & -  \frac{2B}{1 - \frac{ B\rho ^{2}}{2}}  N\eta(t)\sigma_{k}^{ij}\nabla_{i}\frac{ \rho^{2} }{2}\nabla_{j}P.
\end{split}
\end{equation*}
At a point where $P(\cdot,t)$ attains its maximum, we have
\begin{equation*}
\nabla_{j}\frac{M}{h} = -\frac{M}{h}\frac{B\nabla_{j}\frac{ \rho^{2} }{2}}{1 - \frac{ B\rho ^{2}}{2}}
= -\frac{M}{h}\frac{B b_{jm} h_{m}}{1 - \frac{ B\rho ^{2}}{2}}.
\end{equation*}
Due to the inverse concavity of  $(\sigma_{k})^{\frac{1}{k}}$, we have from Lemma 5 in \cite{AMZ2013},
\beqs \left((\sigma_{k})^{\frac{1}{k}}\right)^{ij}b_{im}b_{jm} \geq (\sigma_{k})^{\frac{2}{k}},\eeqs
which means
\beqs \sigma_{k}^{ij}b_{im}b_{jm} \geq k (\sigma_{k})^{1+\frac{1}{k}}.\eeqs
Then at the  point    where $P(\cdot,t)$ attains its maximum, we have
\begin{equation*}
\begin{split}
& \partial_{t}P
 \leq \frac{1}{1 - \frac{ B\rho^{2} }{2}}\left[- \frac{M}{h} \left(k  + \frac{\varphi'h}{\varphi} \right)
 +  \frac{M^{2}}{h^{2}}\eta(t)\left(k + \frac{\varphi'h}{\varphi}\right)  \right] \\
 &+ \frac{M}{h}\frac{B}{\left(1 - \frac{B \rho ^{2}}{2}\right) ^{2}} \left[ (k+1)Nh\eta(t)\sigma_{k}
 - \rho^{2} + \eta(t)\sigma_{k}\nabla_{i}h\nabla_{i}N - kN\eta(t)(\sigma_{k})^{1+\frac{1}{k}}\right].
\end{split}
\end{equation*}
From Lemmas \ref{C0-estimate} and \ref{eta-estimate}, we have
\begin{equation*}
 \partial_{t}P
 \leq c_{1}P +  c_{2}P^{2} - c_{3}P^{2 + \frac{1}{k}}.
\end{equation*}
By maximum principle, we see that $P(x,t)$ is uniformly bounded from above. The upper bound of
$\sigma_{k}$ follows from the uniform bounds on $h$ and $\rho$.
\end{proof}

Now we can derive the  upper bounds of principal curvatures $\kappa_{i}(x,t)$ of $M_t$
for $i=1,\cdots, n-1$.

\begin{lemma}\label{C2-estimate-2}
   Under the assumptions of Theorem \ref{thm1},
we have $$ \kappa_{i} \leq C, \quad \forall (x,t) \in \mathbb{S}^{n-1} \times [0, T),$$
where $C$ is a positive constant independent of $t$.
\end{lemma}
\begin{proof}
By rotation, we assume that the maximal eigenvalue of $b^{ij}$ at $t$  is attained at point $x_{t}$ in the direction
 of the unit vector $e_{1} \in T_{x_{t}} \uS$. We also choose orthonormal vector field such that $b^{ij}$ is diagonal.
 By the evolution equation of $b^{ij}$ in Lemma \ref{evolutions}, we get
\begin{equation*}
\begin{split}
& \partial_{t}\frac{b^{11}}{h} -  N\eta(t)\sigma_{k}^{pq}\nabla_{pq}\frac{b^{11}}{h}  \\
& = \frac{2}{h}N\eta(t)\sigma_{k}^{pq} \nabla_{p}\frac{b^{11}}{h}
\nabla_{q}h + \frac{N}{h^{2}}\eta(t) b^{11}\sigma_{k}^{pq}\nabla_{pq}h
\\
&- (k+1) \frac{N}{h}\eta(t)\sigma_{k}(b^{11})^{2} +  \frac{N}{h}\eta(t)\sigma_{k}^{pq}\delta_{pq}b^{11} \\
& - \frac{N}{h}\eta(t)(b^{11})^{2}(\sigma_{k}^{pq,mn} + 2\sigma_{k}^{pm}b^{nq})\nabla_{1}b_{pq}\nabla_{1}b_{mn}\\
& -\frac{\eta(t)}{h}(b^{11})^{2}( \nabla_{11}N\sigma_{k}  + 2\nabla_{1}\sigma_{k}\nabla_{1}N
) - \frac{b^{11}}{h^{2}}N\eta(t)\sigma_{k}+ \frac{2b^{11}}{h}\\
&= \frac{2}{h}N\eta(t)\sigma_{k}^{pq} \nabla_{p}\frac{b^{11}}{h}
\nabla_{q}h - (k+1) \frac{N}{h}\eta(t)\sigma_{k}(b^{11})^{2} \\
& - \frac{N}{h}\eta(t)(b^{11})^{2}(\sigma_{k}^{pq,mn} + 2\sigma_{k}^{pm}b^{nq})\nabla_{1}b_{pq}\nabla_{1}b_{mn}\\
& -\frac{\eta(t)}{h}(b^{11})^{2}(\nabla_{11}N\sigma_{k}  + 2\nabla_{1}\sigma_{k}\nabla_{1}N)
+ (k-1) \frac{b^{11}}{h^{2}}N\eta(t)\sigma_{k}+ \frac{2b^{11}}{h}.
\end{split}
\end{equation*}
According to inverse concavity of $(\sigma_{k})^{\frac{1}{k}}$, we obtain by  \cite{Urbas1991} or \cite{AMZ2013}
\beqs
(\sigma_{k}^{pq,mn} + 2\sigma_{k}^{pm}b^{nq})\nabla_{1}b_{pq}\nabla_{1}b_{mn} \geq \frac{k+1}{k}\frac{(\nabla_{1}\sigma_{k})^{2}}{\sigma_{k}}.
\eeqs
On the other hand, by Schwartz inequality, the following inequality holds
\beqs
2|\nabla_{1}\sigma_{k}\nabla_{1}N| \leq \frac{k+1}{k}\frac{N(\nabla_{1}\sigma_{k})^{2}}{\sigma_{k}} + \frac{k}{k+1}\frac{\sigma_{k}(\nabla_{1}N)^{2}}{N}.
\eeqs
Hence,  we have at $(x_{t}, t)$
\begin{equation*}
\partial_{t}\frac{b^{11}}{h}\leq  -\frac{(b^{11})^{2}}{h}\sigma_{k}\eta(t)\left[\nabla_{11}N -
\frac{k}{k+1}\frac{(\nabla_{1}N)^{2}}{N} + (k+1)N  + (1-k)\frac{Nb_{11}}{h}\right]+ \frac{2b^{11}}{h}.
\end{equation*}
Let $\tau$ be the arc-length of the great circle  passing through $x_{t}$ with the unit tangent vector $e_{1}$.
Notice that
\begin{equation*}
\nabla_{11}N -\frac{k}{k+1}\frac{(\nabla_{1}N)^{2}}{N} + (k+1)N = (k+1)N^{\frac{k}{k+1}}\left(N^{\frac{1}{k+1}} + (N^{\frac{1}{k+1}})_{\tau\tau}\right).
\end{equation*}
Since
\begin{equation*}
\begin{split}
N_{\tau} & = (f^{-1})_{\tau} \varphi h + f^{-1} \varphi h_{\tau}\left(1+ \frac{\varphi'h}{\varphi} \right)\\
N_{\tau\tau} & = (f^{-1})_{\tau\tau} \varphi h + 2(f^{-1})_{\tau} \varphi h_{\tau}\left(1+ \frac{\varphi'h}{\varphi} \right) + f^{-1} \varphi'h_{\tau}^{2}\left(1 + \frac{\varphi'h}{\varphi} \right)\\
&+ f^{-1} \varphi h_{\tau\tau}\left(1 + \frac{\varphi'h}{\varphi} \right) + f^{-1} \varphi h_{\tau}^{2}\left(1+ \frac{\varphi'h}{\varphi} \right)',
\end{split}
\end{equation*}
here $f^{-1}$ is  $\dfrac{1}{f}$.

We have by direct computations
\begin{equation*}
\begin{split}
& 1+ N^{-\frac{1}{k+1}}\left(N^{\frac{1}{k+1}}\right)_{\tau\tau}  \\
 &= 1 + \frac{1}{k+1}N^{-1}N_{\tau\tau} - \frac{k}{(k+1)^{2}}N^{-2}N_{\tau}^{2}\\
&= 1 + \frac{1}{k+1}f(f^{-1})_{\tau\tau} + \frac{2f}{(k+1)h}(f^{-1})_{\tau}h_{\tau}\left(1+ \frac{\varphi'h}{\varphi} \right) \\
& + \frac{\varphi'}{(k+1)\varphi h}h_{\tau}^{2}\left(1+ \frac{\varphi'h}{\varphi} \right) + \frac{h_{\tau\tau}}{(k+1)h}\left(1+ \frac{\varphi'h}{\varphi} \right)\\
& + \frac{h_{\tau}^{2}}{(k+1)h}\left(1+ \frac{\varphi'h}{\varphi} \right)' - \frac{k}{(k+1)^{2}}f^{2}(f^{-1})_{\tau}^{2}\\
& - \frac{2kf}{(k+1)^{2}h}\left(1+ \frac{\varphi'h}{\varphi} \right)(f^{-1})_{\tau}h_{\tau} - \frac{kh_{\tau}^{2}}{(k+1)^{2}h^{2}}\left(1+ \frac{\varphi'h}{\varphi} \right)^{2}\\
 &= 1 + \frac{1}{k+1}f(f^{-1})_{\tau\tau} + \frac{2f}{(k+1)^{2}h}(f^{-1})_{\tau}h_{\tau}\left(1+ \frac{\varphi'h}{\varphi} \right) \\
&  + \frac{h_{\tau\tau}}{(k+1)h}\left(1+ \frac{\varphi'h}{\varphi} \right) + \frac{h_{\tau}^{2}}{(k+1)h}\left(1+ \frac{\varphi'h}{\varphi} \right)' - \frac{k}{(k+1)^{2}}f^{2}(f^{-1})_{\tau}^{2}\\
&  +\frac{h_{\tau}^{2}}{(k+1)^{2}h^{2}}\left(1+ \frac{\varphi'h}{\varphi} \right)\left(\frac{\varphi'h}{\varphi}  - k\right)\\
&= \frac{1 + \frac{\varphi'h}{\varphi}  }{k+1}\frac{h_{\tau\tau} + h}{h} + \frac{h_{\tau}^{2}}{(k+1)h}\left(1+ \frac{\varphi'h}{\varphi} \right)' \\
& -\frac{1+\frac{\varphi'h}{\varphi} }{h(k+1)^{2}}f
\left[h_{\tau}\left(\frac{k - \frac{\varphi'h}{\varphi}  }{fh}\right)^\frac{1}{2} - (f^{-1})_{\tau}\left(\frac{hf}{k - \frac{\varphi'h}{\varphi}  }\right)^\frac{1}{2} \right]^{2}\\
& + \frac{1}{k+1}\left[\left(k - \frac{\varphi'h}{\varphi} \right) - (f^{-1})_{\tau}^{2}f^{2}\left(\frac{k}{k+1} + \frac{1}{k+1} \frac{1+\frac{\varphi'h}{\varphi} }{\frac{\varphi'h}{\varphi} - k}\right)
+ (f^{-1})_{\tau\tau}f\right]\\
& \geq \frac{1 + \frac{\varphi'h}{\varphi}  }{k+1}\frac{h_{\tau\tau} + h}{h}  \\
& + \frac{1}{k+1}\left[\left(k - \frac{\varphi'h}{\varphi}  \right) - (f^{-1})_{\tau}^{2}f^{2}\left(\frac{k}{k+1} + \frac{1}{k+1} \frac{1+\frac{\varphi'h}{\varphi}}{\frac{\varphi'h}{\varphi}  - k}\right)
+ (f^{-1})_{\tau\tau}f\right],
\end{split}
\end{equation*}
where in the last inequality we use the conditions $\dfrac{\varphi'h}{\varphi}  \leq -1$ and $(\dfrac{\varphi'h}{\varphi} )' \geq 0.$
Since $(k  + 1)f^{-\frac{1}{k + a}}e_{ij} + (k + a)(f^{-\frac{1}{k + a}})_{ij}$ is positive definite and  $
-a \leq \dfrac{\varphi'h}{\varphi}  \leq -1, $ thus we can estimate
\begin{equation*}
\begin{split}
& \left(k - \frac{\varphi'h}{\varphi} \right) - (f^{-1})_{\tau}^{2}f^{2} \frac{k - \frac{\varphi'h}{\varphi}  - 1}{k - \frac{\varphi'h}{\varphi} }
+ (f^{-1})_{\tau\tau}f \\
& \geq k  + 1 - (f^{-1})_{\tau}^{2}f^{2} \frac{k + a - 1}{k + a }
+ (f^{-1})_{\tau\tau}f  \\
& = k  + 1 + (k + a)f^{\frac{1}{k + a}}(f^{-\frac{1}{k + a}})_{\tau\tau}\\
& = f^{\frac{1}{k + a}}\left[(k  + 1)f^{-\frac{1}{k + a}} + (k + a)(f^{-\frac{1}{k + a}})_{\tau\tau}\right]\\
& \geq c_{f},
\end{split}
\end{equation*}
where $c_{f}$ is a positive constant depending on $f$ and the minimal eigenvalue of $(k  + 1)f^{-\frac{1}{k + a}}e_{ij} + (k + a)(f^{-\frac{1}{k + a}})_{ij}$.

Now  we can derive
\begin{equation*}
\partial_{t} \frac{b^{11}}{h}  \leq  -\left(\frac{b^{11}}{h}\right)^{2}N\sigma_{k}\eta(t)( c_{f}h + (2 - a - k)b_{11}) + \frac{2b^{11}}{h}.
\end{equation*}
By the uniform bounds on $h$,  $f$, $\eta$ and $\sigma_{k}$, we conclude
\begin{equation*}
\partial_{t}\frac{b^{11}}{h} \leq  -c_{1}\left(\frac{b^{11}}{h}\right)^{2} + c_{2}\frac{b^{11}}{h}.
\end{equation*}
Here $c_{1}$  and $c_{2}$ are  positive constants independent of  $t$.
The maximum principle then gives the upper bound of  $b^{11} $ and the result follows.
\end{proof}

When $f \equiv 1$,  it can be seen from the proof of Lemma \ref{C2-estimate-2} that the conditions on $f$ and the lower bound of $\dfrac{\varphi'h}{\varphi} $
 can be removed.

\begin{corollary}\label{f=1}
   Under the assumptions of Theorem \ref{thm2},
we have $$ \kappa_{i} \leq C, \quad \forall (x,t) \in \mathbb{S}^{n-1} \times [0, T),$$
where $C$ is a positive constant independent of $t$.
\end{corollary}

Combining Lemma \ref{sigma-lower-estimate}, Lemma \ref{sigma-upper-estimate} and Lemma \ref{C2-estimate-2}, we see that
the principal curvatures of $M_{t}$ has uniform positive upper and lower bounds.
This together with Lemma \ref{C0-estimate} and Corollary \ref{cor3.2} implies that the
evolution equation \eqref{seq} is uniformly parabolic on any finite time
interval. Thus, the result of \cite{KS1980} and the standard
parabolic theory show that the smooth solution of \eqref{seq} exists for all
time.
And by these estimates again, a subsequence of $M_t$ converges in $C^\infty$ to
a positive, smooth, strictly convex hypersurface $M_\infty$ in $\R^n$.
To complete the proofs of Theorem \ref{thm1} and Theorem \ref{thm2}, it only needs to check the support
function of $M_\infty$ satisfies Eq. \eqref{eq}.

\section{Convergence of the flow}

By Lemma \ref{lem1}, Lemma  \ref{C0-estimate} and  Lemma  \ref{sigma-upper-estimate}, the functional
\begin{equation*}
V(t)=\int_{\uS} h(x,t)\sigma_{k}(x,t) \dd x
\end{equation*}
  is non-decreasing along the flow and
\begin{equation} \label{eq4.1}
    V(t) \leq C, \quad \forall t\geq 0.
\end{equation}
This tells that
\begin{equation*}
\int_0^t V'(t) \dd t =V(t)-V(0) \leq V(t) \leq C,
\end{equation*}
which leads to
\begin{equation*}
\int_0^\infty V'(t) \dd t  \leq C.
\end{equation*}
This implies that there exists a subsequence of times $t_j\to\infty$ such that
\begin{equation*}
V'(t_j) \to 0 \text{ as } t_j\to\infty.
\end{equation*}

By Lemma \ref{lem1}, we have
\begin{equation*}
  \begin{split}
    (k+1)V'&(t_j) V(t_j) \\
    &= \int_{\uS}\frac{1}{f(x)}\sigma_{k}^{2}(x)h\varphi(h) \dd x\int_{\uS}\frac{h}{\varphi(h)}f(x)\dd x - \left(\int_{\uS}h\sigma_{k}\dd x\right)^{2}.
  \end{split}
\end{equation*}
Since $h$ and $\sigma_{k}$ have uniform positive upper and lower bounds, by
passing to the limit, we obtain
\begin{equation*}
  \int_{\uS}\frac{1}{f(x)}\widetilde{\sigma_{k}}^{2}(x)\varphi(\tilde{h})\tilde{h}\dd x\int_{\uS}\frac{\tilde{h}}{\varphi(\tilde{h})}f(x) \dd x
  = \left(\int_{\uS}\tilde{h}\widetilde{\sigma_{k}}\dd x\right)^{2},
  \end{equation*}
where $\widetilde{\sigma_{k}}$ and $\tilde{h}$ are  the
$k$-th  elementary symmetric function for principal curvature radii and the support function of $M_\infty$.
According to  the equality condition for the H\"older's inequality, there exists a constant
$c \geq 0$ such that
\begin{equation*}
 c\,\varphi(\tilde{h})  \widetilde{\sigma_{k}}(x) = f \text{ on }\uS.
\end{equation*}
Noticing that $\tilde{h}$ and $\widetilde{\sigma_{k}}$ have positive upper and lower
bounds, $c$ should be positive. The proofs of Theorems \ref{thm1} and \ref{thm2} are finished.
\\

{\bf Acknowledgments } The authors would like to thank the anonymous referee for helpful suggestions.







\begin{thebibliography}{99}
\bibitem{AMZ2013} Andrews, B., McCoy, J. and Zheng, Y.: Contracting convex hypersurfaces by curvature, Calc. Var. Partial Diff. Equ., 47(2013): 611-665.

\bibitem{BLYZ2013}Boroczky, K. J., Lutwak, E., Yang, D. and Zhang, G.: The logarithmic Minkowski problem, J. Amer. Math. Soc., 26 (2013): 831-852.

\bibitem{BIS2019}
 ~Bryan, P., Ivaki, M.~N.  and ~Scheuer, J.:  A unified flow approach to
  smooth, even {$L_p$}-{M}inkowski problems, Anal. PDE, 12 (2019):
  pp.~259--280.

\bibitem{Chen&Huang&Zhao2018} Chen, C. Q., Huang, Y. and Zhao, Y. M.: Smooth solutions to the $L_{p}$ dual Minkowski problems, Math. Ann., 373(2019): 953 - 976.

\bibitem{ChenLi}
Chen,  H. and Li, Q. R.:  The {$L_p$} dual {M}inkowski problem and
  related parabolic flows,  Preprint.

 \bibitem{ChouWang2000} Chou, K. S. and Wang, X. J.:  A logarithmic {G}auss curvature flow and
  the {M}inkowski problem, Ann. Inst. H. Poincar\'e Anal. Non Lin\'eaire, 17
  (2000):~733--751.






 \bibitem{ChouWang2006}Chou, K. S. and Wang, X. J.: The {$L_p$}-Minkowski problem and the Minkowski problem in centroaffine geometry, Adv. Math., 205 (2006): 33-83.


 \bibitem{ChowTsai1997} Chow, B. and Tsai, D. H.: Expansion of convex hypersurfaces by nonhomougeneous functions of curvature, Asian J. Math., 1(1997): 769-784.

\bibitem{GHWXY2019} Gardner, R. J., Hug, D., Weil, W., Xing, S. and Ye, D.: General volumes in the Orlicz-Brunn-Minkowski theory and a related Minkowski problem I, Calc. Var. Partial Diff. Equ., 58 (2019): pp. Art. 12, 35 pp.

\bibitem{GHWXY2020} Gardner, R. J., Hug, D., Weil, W., Xing, S. and Ye, D.: General volumes in the Orlicz-Brunn-Minkowski theory and a related Minkowski problem II,  Calc. Var. Partial Diff. Equ., 59 (2020): pp. Art. 15, 33 pp.



\bibitem{Gerhardt2014} Gerhardt, C.: Non-scale-invariant inverse curvature flows in Euclidean space, Calc. Var. Partial Diff. Equ., 49(2014): 471-489.

 \bibitem{GuanMa2003} Guan, P. F. and Ma, X. N.: The Christoffel-Minkowski problem. I. Convexity of solutions of Hessian equation,
Invent. Math., 151(2003): 553-571.

 \bibitem{GuanXia2018} Guan, P. F. and Xia. C.: {$L^p$} Christoffel-Minkowski problem: the case $1 < p < k+1$. Calc. Var. Partial Diff. Equ.,
 (2018): 57:69.

  \bibitem{Hamilton86} Hamilton, R. S.: Four-manifolds with positive curvature operator, J. Diff. Geom., 24(1986): 153-179.

\bibitem{HLYZ2010}Harbel, C., Lutwak, E., Yang, D. and Zhang, G.: The even Orlicz Minkowski problem, Adv. Math., 224(2010): 2485-2510.

\bibitem{YeLiWang2016}He, Y., Li, Q. R. and Wang, X. J.: Multiple solutions of the {$L_p$}-Minkowski problem, Calc. Var. Partial Diff. Equ., 55 (2016): Art. 117, 13 pp.


 \bibitem{HMS2004} Hu, C., Ma, X. N. and Shen, C.: On Christoffel-Minkowski problem of Firey's $p$-sum,  Calc. Var. Partial Diff. Equ.,
 21(2004): 137-155.

\bibitem{HuangLiuXu2015} Huang, Y., Liu, J. and Xu, L.: On the uniqueness of Lp-Minkowski problems: the constant p-curvature case in $R^{3}$, Adv. Math., 281 (2015): 906-927.

 \bibitem{HuangActa2016} Huang, Y., Lutwak, E., Yang, D. and Zhang, G.: Geometric measures in the dual Brunn-Minkowski theory and their associated Minkowski problems, Acta Math., 216 (2016): 325-388.

 \bibitem{HLYZ2005}  Hug, D., Lutwak, E., Yang, D. and Zhang, G.: On the Lp Minkowski problem for polytopes, Discrete Comput. Geom., 33 (2005): 699-715.

\bibitem{Ivaki2019}  Ivaki, M. N.: Deforming a hyper surface by principal radii of curvature and support function, Calc. Var. Partial Diff. Equ., 58(1)(2019): 2133-2165.

\bibitem{JianLu2019} Jian, H. Y. and Lu, J., Existence of solutions to the Orlicz-Minkowski problem, Adv. Math., 344(2019): 262-288.

\bibitem{JianLuWang2015}Jian, H. Y., Lu, J. and Wang, X. J.: Nonuniqueness of solutions to the $L_{p}$-Minkowski problem, Adv. Math., 281 (2015): 845-856.

\bibitem{JianLuZhu2016} Jian, H. Y., Lu, J. and Zhu, G.: Mirror symmetric solutions to the centro-affine Minkowski problem, Calc. Var. Partial Diff. Equ., 55 (2016): Art. 41, 22 pp.

 \bibitem{KS1980}
Krylov, N. V.  and Safonov, M. V.: A property of the solutions of
  parabolic equations with measurable coefficients, Izv. Akad. Nauk SSSR Ser.
  Mat., (44) 1980: 161--175, 239.

\bibitem{LiShengWang2020} Li, Q. R., Sheng, W. and Wang, X. J.:  Flow by {G}auss curvature to the
  {A}leksandrov and dual {M}inkowski problems, J. Eur. Math. Soc. (JEMS), 22
  (2020): 893--923.

\bibitem{LiuLuTrans} Liu, Y. N. and Lu, J.:
A flow method for the dual {O}rlicz-{M}inkowski problem, Trans. Amer. Math. Soc., 373 (2020): 5833--5853.

\bibitem{LiuLu2020} Liu, Y. N. and Lu, J.: A generalized Gauss curvature flow related to the Orlicz-Minkowski problem. arXiv:2005.02376.

\bibitem{Lu2018} Lu, J.: Nonexistence of maximizers for the functional of the centroaffine Minkowski problem. Sci. China Math., 61(2018): 511-516.

\bibitem{LuWang2013} Lu, J. and Wang, X. J.: Rationally symmetric solutions to the $L_{p}$-Minkowski problem, J. Diff. Equ., 254(2013): 983-1005.

\bibitem{Lutwak1993} Lutwak, E.: The Brunn-Minkowski-Firey theory. I. Mixed volumes and the Minkowski problem, J. Differential Geom., 38 (1993): 131-150.

 \bibitem{Schneider} Schneider, R.: Convex bodies, the Brunn-Minkowski theory, vol. 151 of Encyclopedia of Mathematics and its Applications, Cambridge University Press, Cambridge, expanded, 2014.

\bibitem{ShengYi2020} Sheng,  W. M. and Yi, C. H.:  A class of anisotropic expanding curvature flow. Disc.  Conti. Dynam. Systems - A. 40(2020): 2017-2035.

\bibitem{SunLong2015} Sun, Y. and Long, Y.: The planar Orlicz Minkowski problem in the $L^{1}$ -sense, Adv. Math., 281(2015): 1364-1383.

\bibitem{Urbas1991} Urbas, J.: An expansion of convex hypersurfaces. J. Diff. Geom., 33(1991): 91-125.

\bibitem{XiJinLeng2014} Xi, D., Jin, H. and Leng, G.: The Orlicz Brunn-Minkowski inequality, Adv. Math., 260(2014): 350-374.

\bibitem{Xia2016} Xia, C.: Inverse anisotropic curvature flows from convex hypersurfaces, J. Geom. Anal., (27)2016: 1-24.

\bibitem{Zhu2014} Zhu, G.: The logarithmic {M}inkowski problem for polytopes. Adv. Math., 262 (2014): ~909--931.

\bibitem{ZouXiong2014} Zou, D. and Xiong, G.: Orlicz-John ellipsoids, Adv. Math., 265(2014): 132-168.

\end{thebibliography}
\end{document}